\documentclass[10pt]{amsart}

\usepackage{amssymb,mathrsfs,mathscinet}
\usepackage{algorithmic}
\usepackage{hyperref}

\usepackage{enumerate}        

\newtheorem{theorem}{Theorem}[section]
\newtheorem{question}[theorem]{Question}
\theoremstyle{definition}

\newtheorem{proposition}[theorem]{Proposition}
\newtheorem*{mainquestion}{Main Question}

\newcommand {\Z}{\mathbb{Z}}

\makeatletter
\def\imod#1{\allowbreak\mkern10mu({\operator@font mod}\,\,#1)}
\makeatother

\begin{document}

\title{On Restricting Subsets of Bases in Relatively Free Groups}
\author[L.~Sabalka]{Lucas Sabalka}
\address{
  Department of Mathematics and Computer Science\\
  Saint Louis University}
\email{\href{mailto:lsabalka@slu.edu}{lsabalka@slu.edu}}

\author[D.~Savchuk]{Dmytro Savchuk}
\address{
  Department of Mathematics and Statistics\\
  University of South Florida}
\email{\href{mailto:savchuk@usf.edu}{savchuk@usf.edu}}

\begin{abstract}

Let $G$ be a finitely generated free, free abelian of arbitrary exponent, free nilpotent, or free solvable group, or a free group in the variety $\mathbf A_m\mathbf A_n$, and let $A = \{a_1, \dots, a_r\}$ be a basis for $G$.  We prove that, in most cases, if $S$ is a subset of a basis for $G$ which may be expressed as a word in $A$ without using elements from $\{a_{l+1},\ldots,a_r\}$, then $S$ is a subset of a basis for the relatively free group on $\{a_1, \dots, a_{l}\}$.

\end{abstract}

\maketitle
\section{Introduction}

Let $F = F(A)$ be a free group of rank $r \geq 2$ with basis $A = \{a_1, \dots, a_r\}$.  A subgroup $V$ of $F$ is \emph{fully invariant} if, for all endomorphisms $f: F \to F$, $f(V) \subseteq V$.  Fully invariant subgroups of $F$ include the trivial subgroup, the derived subgroups, the subgroups in the lower central series of $F$, and, more generally, verbal subgroups of $F$.  Every fully invariant subgroup is normal in $F$, and the quotients of the form $F/V$ include the free group, free nilpotent groups, and free solvable groups.  A subset $S \subset F$ is \emph{primitive in $F$ mod $V$} if the corresponding set of cosets $SV$ of $V$ can be extended to a basis of $F/V$ (i.e. to a generating set for $F/V$ such that each mapping of this set to $F/V$ extends to an endomorphisms of $F/V$).

We wish to study the following situation.  Let $V$ be a fully invariant subgroup of $F$ and let $S$ be a primitive set mod $V$. Assume that there exists some index $l \in \{1, \dots, r-1\}$ such that, for each $s \in S$, the generators $a_t, t=l+1,\ldots,r$ of $F(A)$ are not used to express $s$ as a reduced word in the alphabet $A$.  In short, $S \subset \hat{F} \subset F$, where $\hat{F} := F(\{a_1, \ldots, a_{l}\})$. Then $\hat V=V\cap\hat F$ is a fully invariant subgroup of $\hat F$ (as every endomorphism of $\hat F$ can be extended to an endomorphism of $F$). Our main question is:

\begin{mainquestion}
Under what conditions is it true that $S$ is also primitive in $\hat{F}$ mod $\hat V$?
\end{mainquestion}

If $V$ is trivial -- that is, if we are considering $S$ to be primitive in $F$ -- and $|S| = r-1$ (where $|S|$ denotes the size of set $S$), then it is not difficult to show that $S$ must be primitive in $\hat{F}$.  The question becomes more interesting when $|S|$ is allowed to vary, and when $V$ is allowed to vary.  Our main results answer this question for different choices of $V$, $|S|$ and $l$.  Let $\mathbf A_m$ denote the variety of all abelian groups whose exponent divides $m$. In particular, $\mathbf A:=\mathbf A_0$ is the variety of all abelian groups. Then the product variety $\mathbf A_m\mathbf A_n$ consists of extensions of abelian groups of exponent dividing $m$ by abelian groups of exponent dividing $n$.

\pagebreak

\begin{theorem}\label{thm:main}  Let $S \subset \hat{F}$ be primitive in $F$ mod $V$.  Then $S$ is primitive in $\hat{F}$ mod $\hat{V}$ if $F/V$ is:
\begin{enumerate}
\item free,
\item free abelian of arbitrary exponent (including 0),
\item free nilpotent,
\item free in the abelian-by-abelian variety $\mathbf A_m\mathbf A_n$ and at least one of the following conditions fails:  $(a)$ $m=0$, $(b)$ $n>0$, or $(c)$ either $|S| = r-1$ or $|S|=l-1$; or
\item free solvable and $|S|=r-1$.
\end{enumerate}
\end{theorem}

The motivation for this theorem came from the statement in the free case, which was initially proved by the authors as a lemma pertaining to the work in \cite{SabalkaSavchuk}.  The result was a generalization of one case of a result describing the structure of subwords of primitive elements in $F$ that do not involve one of the generators of $F$.  Though it was eventually realized that this lemma was unnecessary for the paper \cite{SabalkaSavchuk}, the authors felt the lemma and proof interesting enough to generalize and publish.

We conclude the introduction with the following natural question.

\begin{question}
Does there exist a variety of groups for which the Main Question has a negative answer?  In particular, are there varieties where a stably primitive element is not necessarily primitive?
\end{question}

The structure of the paper is as follows. We begin by examining the case where $V$ is trivial in Section~\ref{sec:free}. Then we proceed to cases of the free abelian groups of arbitrary exponent in Section~\ref{sec:abelian} and the free nilpotent and free solvable groups in Section~\ref{sec:nilpotent}. We conclude the paper by the cases of free abelian-by-abelian of arbitrary exponents in Section~\ref{sec:product_variety} and free solvable groups in Section~\ref{sec:solvable}.

The authors would like to thank Mark Sapir for helpful advice on this paper, and Michael Handel, Ilya Kapovich, Olga Kharlampovich, Marcin Mazur, and Benjamin Steinberg for useful conversations on this material. And we are especially grateful to Martin Kassabov for providing a short proof of the free nilpotent case that replaced a longer argument involving `lifting primitivity' and to Vitali\u{\i} Roman'kov for pointing out an error in an earlier version of the paper. Finally, we greatly appreciate the comments and suggestions of the referee that enhanced the paper.

\section{The Free Case}\label{sec:free}

In the free case, the answer to our Main Question is:  always.  Our proof uses an interpretation of primitivity using Fox derivatives via a theorem of Umirbaev.

We begin by recalling the definition of free Fox derivatives and establish notation for them.  Let $\mathbb ZF_r$ denote the integral ring over the free group $F_r$. For each $j=1,2,\ldots,r$ define a linear map $\partial_j\colon \Z F_r\to\Z F_r$ recursively by
\[\partial_j(a_j)=1,\qquad \partial_j(a_i)=0, i\neq j\]
and
\[\partial_j(uv)=\partial_j(u)+u\partial_j(v) \text{ for all } u,v\in F_r.\]
The map $\partial_j$ is called the (left) free Fox derivative associated with $a_j$.

\begin{theorem}\cite{umirbaev:primitive}\label{thm:Umirbaev}
A subset $\{x_1,\ldots,x_k\}$ is primitive in $F_r$ if and only if the $k\times r$ Jacobian matrix $J=\bigl[\partial_j(x_i)\bigr]_{1\leq i\leq k, 1\leq j\leq r}$, is right invertible in the integral ring $\mathbb ZF_r$.
\end{theorem}

We note that the ``only if'' part in this theorem was proven by Birman in \cite{birman:primitive}.

\begin{theorem} \label{thm:free}
Let $S=\{x_1,x_2,\ldots, x_k\}\subset \hat{F}$ be primitive in $F$. Then $S$ is primitive in $\hat{F}$.
\end{theorem}

\begin{proof}

By Theorem \ref{thm:Umirbaev}, since the set $S$ is primitive in $F$, the associated Jacobian $J$ is such that there exists an $r\times k$ matrix $P=[p_{jl}]$ with $p_{jl}\in \mathbb ZF_r$
satisfying
    \begin{equation}
    \label{eqn_JP}
    JP=I_k,
    \end{equation}
where $I_k$ is the $k\times k$ identity matrix over $\mathbb ZF_r$.

For $m > l$, since elements of $S$ do not involve $a_m$, the $m^{th}$ column of the matrix $J$ consists of zeros and all other entries of $J$ do not involve $a_m$. Each entry of matrix $P$ can be uniquely written in the form $p_{ij}=q_{ij}+r_{ij}$, where $q_{ij}$ represents the sum of all terms in $p_{ij}$ involving some $a_m$ with $m>l$, and $r_{ij}$ is an element of $\mathbb ZF_{l}$ not involving any $a_m$ with $m>l$. Then for matrices $Q=[q_{ij}]$ and $R=[r_{ij}]$ we have $P=Q+R$, and by Equation~\eqref{eqn_JP}, $JQ+JR=I_k$. But each entry of matrix $JQ$ is either $0$ or involves $a_m$ for some $m>l$. Since neither $I_k$ nor $JR$ involves $a_m$ for $m>l$ we must have that $JQ=0$, yielding
    \begin{equation}
    \label{eqn_JR}
    JR=I_k
    \end{equation}
Now let $\hat J$ and $\hat R$ be the matrices obtained from $J$ and $R$, respectively, by deleting last $r-l$ columns and rows, respectively.  Then $\hat J$ is the Jacobian matrix of the set $S$ seen as a subset of $\hat F=F_{l}$ and $\hat R$ is a matrix over $\mathbb Z\hat F=\mathbb ZF_{l}$.  Also, equation~\eqref{eqn_JR} implies that
    \begin{equation*}
    \hat J\hat R=I_{k}.
    \end{equation*}
Thus by Umirbaev's criterion the set $S$ is a subset of the basis of $\hat{F}$.

\end{proof}

Via private communication, Ilya Kapovich has detailed an alternative proof of Theorem \ref{thm:free}, using Gersten's characterization of Whitehead's algorithm for subgroups \cite{Gersten}.  Olga Kharlampovich has also suggested a proof using Bass-Serre theory. The benefit of the above proof is that it generalizes to the case of free abelian-by-abelian groups of arbitrary exponents in Section~\ref{sec:product_variety} (and also to the case of free metabelian groups and partially to the case of free solvable groups, though shorter proofs in these cases are given here).

\section{The Abelian Case}\label{sec:abelian}

The answer to our Main Question in the free abelian case is also:  always.  Our proof works for free abelian groups of arbitrary exponent and uses linear algebra.

For any $n> 0$, Let $\Z_n := \Z/n\Z$ denote the free group of rank 1 in the variety $\mathbf A_n$ of all abelian groups of exponent dividing $n$, and let $F^n$ denote the subgroup of $F$ generated by all $n^{th}$ powers and similarly define $\hat{F}^n$.  Let $F'$ denote the commutator subgroup of $F$ and similarly define $\hat{F}'$.  Note $(\hat F)'=\widehat{(F')}$, so $\hat F'$ is well defined. Further, for each $x\in F$ we denote by $\bar x$ the image of $x$ in $F/F'$ under the canonical epimorphism.  For a subset $A$ of $F$ define $\bar A=\{\bar a\colon a\in A\}$.

We start with the proof for free abelian groups in the case of exponent $0$.

\begin{proposition}\label{prop:abelian}
Let $S=\{x_1,x_2,\ldots, x_k\}\subset \hat F$ be primitive in $F \mod F'$. Then $S$ is primitive in $\hat F\mod \hat F'$.
\end{proposition}

Before we proceed to our proof, we note that this Proposition follows from a well-known fact (see, for example,~\cite{kargapolov_m:fundamentals_book_eng}) that the basis of any subgroup of a finitely generated free abelian group can be extended to a basis of the whole group. Indeed, since $\bar S$ can be extended to a basis in $F/F'$, we get that $\bar S$ is linearly independent. Hence, $\bar S$ is a basis of a subgroup $\langle \bar S\rangle<\hat F/\hat F'$. We give a proof of Proposition~\ref{prop:abelian} that extends to the case of free abelian groups of finite exponent in Theorem~\ref{thm:abelian_exponent}.

\begin{proof}
There is a canonical embedding $\Z^{l} \cong \hat F/\hat F'\hookrightarrow F/F' \cong \Z^r$ and with a slight abuse of notation we will sometimes identify $\hat F/\hat F'$ with its image under this embedding.

The standard basis of $\Z^r\cong F/F'$ is $\{\bar a_1, \dots, \bar a_r\}$, whereas $\{\bar a_1, \dots, \bar a_{l}\}$ is the standard basis of $\Z^{l} \cong \hat F/\hat F'$.  Extend $\bar S$ to a basis $B$ of $\Z^r$.  Express $B$ with respect to the standard basis as a matrix $M$, where the elements of $\bar S$ correspond to the first $k$ columns of $M$.  Let $\hat M$ denote the submatrix of $M$ obtained by deleting the last $r-l$ rows and columns.  Then $\hat M$ corresponds to a set $\hat B$ of $l$ elements of $\Z^{l}$. Since the elements of $\bar S$ do not involve generators $\bar a_{l+1},\ldots,\bar a_r$, the matrix $M$ has a block triangular shape:
\[M=\left[\begin{array}{cc}
\hat M&P\\
0&Q\\
\end{array}\right],\]
for some $l\times (r-l)$-matrix $P$ and $(r-l)\times(r-l)$-matrix $Q$. Then since $M$ is in $SL_r(\mathbb Z)$ we have
\[\det\hat M\cdot \det Q=\det M=\pm1.\]
This is possible only if $\det\hat M=\pm 1$, and $\hat M$ is invertible.  This shows $\hat B$ is a basis of $\Z^{l} \cong \hat{F}/\hat{F}'$.  But $\bar S \subset \hat B$, so $\bar S$ is primitive in $\hat{F}/\hat{F}'$ and thus $S$ is primitive in $\hat F\mod \hat F'$.
\end{proof}

Now we proceed to the case of free abelian groups with exponent $n > 0$. First we note that the free group of rank $r$ in this variety is isomorphic to $F/(F'F^n)$.

\begin{theorem}\label{thm:abelian_exponent}
Let $S=\{x_1,x_2,\ldots, x_k\}\subset \hat F$ be primitive in $F \mod F'F^n$. Then $S$ is primitive in $\hat F\mod \hat F'\hat F^n$.
\end{theorem}

\begin{proof}
The proof is almost identical to that of Proposition~\ref{prop:abelian}.  We replace $F'$ with $F'F^n$, $\hat F'$ with $\hat F'\hat F^n$, and $\Z$ with $\Z_n$, but otherwise define all other objects in the same way.  The only other change to make is the following argument to show that $\hat M$ is invertible.  The matrix $M$ is invertible in $SL_r(\Z_n)$ if and only if $\det M$ is a unit in $\Z_n$. But a product of elements in $\Z_n$ is a unit if and only if each of the elements is also a unit.  Thus, $\det M = \det \hat M \cdot \det Q$ is a unit in $\Z_n$ if and only if $\det \hat M$ is a unit in $\Z_n$, if and only if $\hat M$ is invertible in $SL_r(\Z_n)$.
\end{proof}

\section{The Free Nilpotent Case}
\label{sec:nilpotent}

We now turn to the free nilpotent and free solvable cases.  A \emph{commutator} of \emph{weight} $c$ is an expression recursively defined by $[y_1, \dots, y_{c-1}, y_c] := [[y_1, \dots, y_{c-1}], y_c]$ and $[y_1, y_2] := y_1^{-1}y_2^{-1}y_1y_2$.  Fix a value for $c\geq2$, and let $\gamma_c(F)$ be the normal closure of subset of $F$ consisting of all commutators of weight $c$ (i.e. the $c^{th}$ term in the lower central series of $F$).  Then $F/\gamma_{c+1}(F)$ is the free nilpotent group of class $c$.  Also define $\gamma_1(F)=F$. Further, we write $F^{(t)}$ for the $t^{th}$ derived subgroup of $F$. Then $F/F^{(t)}$ is a free solvable group of derived length $t$.  Note that $F' = \gamma_2(F) = F^{(1)}$, so the free abelian group is the free nilpotent group of class $1$ and the free solvable group of derived length $1$.

\begin{theorem}\label{thm:nilpotent}~\\\vspace{-.4cm}
\begin{enumerate}
\item Let $S=\{x_1,x_2,\ldots, x_k\}\subset \hat F$ be primitive in $F \mod \gamma_{t+1}(F)$ for some $t\geq 1$. Then $S$ is primitive in $\hat F\mod \gamma_{t+1}(\hat F)$.
\end{enumerate}
\end{theorem}

The following proof is based on suggestions and references from Martin Kassabov and the referee.  In the nilpotent case, our original proof used a primitivity lifting technique \cite{gupta_g:lifting_primitivity_nilpotent}.  In the solvable case, our original proof was based on a generalization of our approach in the free case, and used Fox derivatives and criteria for primitivity by Roman'kov and Timoshenko  in the free metabelian case \cite{romankov:primitivity,timoshenko:algorithmic}, and by Krasnikov in the free solvable case~\cite{krasnikov:generating_elements_eng}.  However, our original proofs were longer and less general.

\begin{proof}
It is a well-known fact (see, for example,~\cite[Lemma 5.9]{magnus_ks:combinatorial_group_theory}) that a set of elements generates a nilpotent group $N$ if and only if the projection of these elements to the abelianization $N/N'$ generates the abelianization. Assume a subset $S = \{x_1, \dots, x_k\}$ is primitive in $F$ mod $\gamma_{t+1}(F)$ and $S$ does not involve generators $a_{l+1},\ldots,a_r$. Then $S$ is also primitive over $F'$, which implies by Proposition~\ref{prop:abelian} that $S$ is primitive in $\hat F$ mod $\hat F'$. Now by the above fact $S$ must be primitive in $\hat F\mod \gamma_{t+1}(\hat F)$.
\end{proof}

\section{The Abelian-by-Abelian Case}
\label{sec:product_variety}

For integers $m, n \geq 0$, we consider the case of free groups in the variety $\mathbf A_m \mathbf A_n$ of all extensions of abelian groups of exponent dividing $m$ by abelian groups of exponent dividing $n$.  Let $V_{m,n}$ and $\hat{V}_{m,n}$ denote the corresponding fully invariant subgroups of $F$ and $\hat F$, respectively. The free Fox derivatives $\partial_i$ induce Fox derivatives $\partial^0_i\colon \Z F\to \Z_m(F/F'F^n)$  (see Gupta and Timoshenko~\cite{gupta_t:primitivity_in_AnAm96}) by
\[\Z F\stackrel{\partial_i}{\longrightarrow}\Z F\stackrel{\alpha^*}{\longrightarrow}\Z (F/F'F^n)\stackrel{\gamma^*}{\longrightarrow}\Z_m(F/F'F^n),\]
where $\alpha^*\colon\Z F\to\Z(F/F'F^n)$ and $\gamma^*\colon \Z(F/F'F^n)\to\Z_m(F/F'F^n)$ are linear extensions of the canonical epimorphisms $\alpha\colon F\to F/F'F^n$ and $\gamma\colon \Z\to\Z_m$ respectively.

In \cite{gupta_t:primitive_systems_in_AmAn99}, Gupta and Timoshenko prove the following theorem.

\begin{theorem}\cite{gupta_t:primitive_systems_in_AmAn99}\label{thm:AmA}
Let $S=\{x_1,x_2,\ldots, x_k\}\subset F$.  Let $J = (\partial^0_j x_i), i=1,\ldots,k, j=1,\ldots,r$ denote the $k \times r$ Jacobian matrix of $S$ over $\Z_m(F/F'F^n)$.  Assume that at least one of the conditions $m=0$, $n>0$, or $k=r-1$ fails. Then $S$ is primitive mod $V_{m,n}$ if and only if:
\begin{enumerate}
\item[$(i)$] the ideal in the ring $\Z_m(F/F'F^n)$ generated by $k^{th}$ order minors of $J$ is the whole ring, and
\item[$(ii)$] $S$ is primitive mod $F'F^n$.
\end{enumerate}
\end{theorem}

Using the above theorem we deduce:
\begin{theorem}
Let $S=\{x_1,x_2,\ldots, x_k\}\subset \hat F$ be primitive in $F \mod V_{m,n}$ and assume at least one of the following conditions fails:
$(a)$ $m=0$, $(b)$ $n>0$, or $(c)$ either $k = r-1$ or $k=l-1$. Then $S$ is primitive in $\hat F\mod \hat V_{m,n}$.
\end{theorem}

\begin{proof}
Let $m_1, \dots, m_t$ denote the $k^{th}$ order minors of $J$ in $\Z_m(F/F'F^n)$, and let $J_m$ denote the ideal of $\Z_m(F/F'F^n)$ generated by $m_1, \dots, m_t$.  As $S \subset \hat F$, $J$ is a matrix over $\Z_m(\hat{F}/\hat{F}'\hat F^n)$ whose last $r-l$ columns are zero columns. Therefore the $l\times k$ Jacobian matrix $\hat J=(\partial_j x_i)$ formed by viewing $S$ as a subset of $\hat F$ is obtained from $J$ by simply removing last $r-l$ columns.  But removing columns of zeros does not change the $k^{th}$ order minors of a matrix, so the $k^{th}$ order minors of $\hat{J}$ are also $m_1, \dots, m_t$, considered now as elements in $\Z_m(\hat{F}/\hat{F}' \hat{F}^n)$ (via the canonical embedding of $\Z_m(\hat{F}/\hat{F}'\hat F^n)$ into $\Z_m(F/F'F^n)$).  Let $\hat{J}_m$ denote the ideal generated by $m_1, \dots, m_t$ in the ring $\Z_m(\hat{F}/\hat{F}'\hat{F}^n)$.

Since $S$ is primitive in $F \mod V_{m,n}$, by Theorem \ref{thm:AmA} we know that: ($i$) $J_m = \Z_m(F/F'F^n)$, and ($ii$) $S$ is primitive mod $F'F^n$.  It follows from ($i$) that $J_m$ contains the identity in $\Z_m(F/F'F^n)$.  Therefore for some $p_i\in\Z_m(F/F'F^n)$ we have
\begin{equation}
\label{eqn:ideal}
1=\sum_{i=1}^{t}m_ip_i.
\end{equation}
We can decompose each $p_i$ as
\[p_i=q_i+r_i,\]
where $q_i$ is the sum of terms in $p_i$ that do not involve $a_{l+1},\ldots,a_r$, and $r_i$ is the sum of terms in $p_i$ that do involve at least one of the $a_{l+1},\ldots,a_r$ in non-zero power. Then equation~\eqref{eqn:ideal} can be rewritten as
\begin{equation*}
\label{eqn:ideal2}
1-\sum_{i=1}^{t}m_iq_i=\sum_{i=1}^{t}m_ir_i
\end{equation*}
The left-hand side of the above equation does not have terms involving $a_{l+1},\ldots,a_r$.  However, since every term of each $m_i$ does not involve any of $a_{l+1}, \dots, a_r$ but every term of each $r_i$ does, every term of $m_ir_i$ involves at least one of the elements $a_{l+1}, \dots, a_r$.  Some of the terms of $m_ir_i$ might cancel, but it follows that every surviving term of $m_ir_i$ must involve at least one of the elements $a_{l+1}, \dots, a_r$.  Further, when we take the sum of $m_ir_i$ we cannot create terms that do not involve $a_{l+1},\ldots,a_r$ as the support of $\sum_{i=1}^{t}m_ir_i$ viewed as a function from $F/F'F^n$ to $\Z_m$ is included into the union of supports of $m_ir_i$, which, in turn, are included in the set of all elements of $F/F'F^n$ that involve $a_{l+1},\ldots,a_r$. Thus, for the two sides to be equal, the only possibility is if both sides are equal to zero.  But this yields that
\begin{equation*}
1=\sum_{i=1}^{t}m_iq_i
\end{equation*}
with $m_i, q_i\in \Z_m(\hat{F}/\hat{F}'\hat{F}^n)$. Thus the identity in $\Z_m(\hat{F}/\hat{F}'\hat{F}^n)$ belongs to the ideal $\hat{J}_m$, and so condition $(i)$ holds for $S$ with respect to the ring $\Z_m(\hat{F}/\hat{F}'\hat{F}^n)$.

It follows from ($ii$), the fact that $S \subset \hat F$, and Proposition~\ref{prop:abelian} and Theorem~\ref{thm:abelian_exponent} that $S$ is primitive mod $\hat{F}'\hat{F}^n$.  Thus, applying Theorem~\ref{thm:AmA} to $S$ thought of as a subset of $\hat F$, we see that $S$ is primitive in $\hat{F} \mod \hat{V}_{m,n}$. Note that Theorem~\ref{thm:AmA} applies in this situation if at least one of the conditions $m=0$, $n>0$, or $k=rank(\hat F)-1=l-1$ fails, which is true by assumption.
\end{proof}

Vitali\u{\i} Roman'kov suggested that one could potentially use the basis cofinality property of free groups of countable rank in subvarieties $\mathbf B$ of $\mathbf N_c \mathbf A$ (where $\mathbf N_c$ denotes the variety of nilpotent groups of class $c$) to generalize our results in this section. This property was proved in this case by Bryant and Roman'kov in~\cite{bryant_r:automorphism_group_of_rel_free_groups99} (see also~\cite{bryant_e:small_index_property97} for definitions). To do so, one must show that primitive systems in a finitely generated relatively free group in $\mathbf B$ that miss certain generators can be lifted to primitive systems in $F$ with the same property. For arbitrary varieties this statement is false, but it could be true for the varieties above.

\section{The Free Solvable Case}
\label{sec:solvable}

The final case we consider is the case of free solvable groups. First of all, the free metabelian case is covered by Section~\ref{sec:product_variety}. Alternatively, for the free metabelian case one can use the primitivity criteria developed by Timoshenko~\cite{timoshenko:on_the_inclusion_elements_into_a_basis_of_free_metabelian_group88,timoshenko:primitive_systems98}
and Roman'kov~\cite{romanov:primitive_elements91,romankov:primitivity}, which use Fox derivatives similar in spirit to Umirbaev's criterion~\cite{umirbaev:primitive} stated in Theorem~\ref{thm:Umirbaev}.
Unfortunately, there is no such criterion for free solvable groups of arbitrary derived length. However, Krasnikov~\cite{krasnikov:generating_elements} has obtained a result for groups of the form $F/[N,N]$ where $N$ is a normal subgroup of $F$, which works only in the case $k=r$.  The Fox derivatives used here are the free Fox derivatives of Section \ref{sec:free}.

\begin{theorem}\cite{krasnikov:generating_elements}\label{thm:krasnikov}
A set $\{x_1,\ldots,x_r\}$ in $F/[N,N]$ is a generating set of this group if and only if the $r\times r$ Jacobian matrix $J=\bigl[D_j(x_i)\bigr]_{1\leq i,j \leq r}$ of free Fox derivatives, projected to $Z(F/N)$, is left invertible in the integral ring $\mathbb Z(F/N)$.
\end{theorem}

With this criterion we obtain an analog of Theorem~\ref{thm:free} in the case of free solvable groups with a restriction on the size of $S$.

\begin{theorem}\label{thm:solvable}
Let $S=\{x_1,x_2,\ldots,x_{r-1}\} \subset \hat{F}$ be primitive in $F$ mod $F^{(t)}$, where $F^{(t)}$ is the $t$-th derived subgroup of $F$, $t > 0$.  Then $S$ is a basis for $\hat{F}/\hat{F}^{(t)}$.
\end{theorem}

\begin{proof}
Since $S$ is primitive in $F\mod F^{(t)}$, we can extend $S$ to a set $\tilde S=\{x_1,x_2,\ldots,x_r\}\subset F$ such that the images of elements of $\tilde S$ in $F/F^{(t)}$ under the canonical epimorphism form a basis. By Theorem~\ref{thm:krasnikov} we have that the associated to $\tilde S$ Jacobian matrix $J=\bigl[D_j(x_i)\bigr]_{1\leq i, j \leq r}$ is left invertible over $\mathbb{Z}(F/F^{(t-1)})$, so there exists an $r\times r$ matrix $P=\bigl[p_{ij}]_{1 \leq i, j \leq r}$ with entries over $\mathbb Z(F/F^{(t-1)})$ such that
    \[PJ=I_r,\]
where $I_r$ is an $r\times r$ identity matrix over $\mathbb Z(F/F^{(t-1)})$.

Since elements of $S$ do not involve the generator $a_r$, we have
    \[D_r(x_j)=0, \quad j=1,\ldots,r-1,\]
and $D_i(x_j)$ does not involve $a_r$ for $1\leq j\leq r-1$. Therefore, matrix $J$ can be written in the form
    \[J=\left[
    \begin{array}{c|c}
    \hat J&\begin{array}{c}0\\ \vdots\\ 0
    \end{array}\\\hline
    D_1(x_r)\cdots D_{r-1}(x_r)&D_{r}(x_r)
    \end{array}\right],\]
where $\hat J=\bigl[D_i(x_j)\bigr]_{1\leq i, j\leq r-1}$ is the Jacobian martix of $S$ in $\hat F/\hat F^{(t-1)}$.

Then
    \[I_r=PJ=\left[
    \begin{array}{c|c}
    *&\begin{array}{c}p_{1r}D_r(x_r)\\ \vdots\\ p_{r-1,r}D_r(x_r)
    \end{array}\\\hline
    *\cdots *&p_{rr}D_{r}(x_r)
    \end{array}\right].\]
Hence, $p_{rr}D_r(x_r)=1$ in $\mathbb Z(F/F^{(t-1)})$ and, in particular, $D_r(x_r)\neq 0$. On the other hand,
    \[p_{in}D_r(x_r)=0, \quad i=1,\ldots,r-1,\]
so we must have $p_{in}=0$ in $\mathbb Z(F/F^{(t-1)})$. This follows from the fact that the free solvable group $F/F^{(t-1)}$ is a torsion-free elementary amenable group, and for such groups the Kaplansky conjecture on zero divisors holds true~\cite{linnell:zero_divisors}. In particular, this implies that the ring $\mathbb Z(F/F^{(t-1)})$ does not have zero divisors.

Now we have

    \begin{eqnarray*}
    I_r=PJ
    &=&\left[\begin{array}{c|c}
        \hat P&\begin{array}{c}0\\ \vdots\\0
        \end{array}\\\hline
        p_{r1}\cdots p_{r,r-1}&p_{rr}
    \end{array}\right] \left[\begin{array}{c|c}
        \hat J&\begin{array}{c}0\\ \vdots\\ 0
        \end{array}\\\hline
        D_1(x_r)\cdots D_{r-1}(x_r)&D_{r}(x_r)
    \end{array}\right]\\
    &=&\left[\begin{array}{c|c}
        \hat P\hat J&\begin{array}{c}0\\ \vdots\\ 0
        \end{array}\\\hline
        *\cdots *&p_{rr}D_{r}(x_r)
    \end{array}\right].
    \end{eqnarray*}

Therefore we must have $\hat P\hat J=I_{r-1}$, where $I_{r-1}$ is the $(r-1)\times(r-1)$ identity matrix over $\mathbb Z(F/F^{(t-1)})$. Similarly to the proof in the free case, it follows that $\hat P$ can be chosen to be in $\mathbb Z(\hat F/\hat F^{(t-1)})$. Finally, applying Theorem~\ref{thm:krasnikov} again, we obtain that $S$ is a basis for $\hat F/\hat F^{(t)}$.
\end{proof}


\def\cprime{$'$} \def\cprime{$'$} \def\cprime{$'$} \def\cprime{$'$}
  \def\cprime{$'$}

\end{document}